\documentclass[a4paper,11pt]{article}
 \usepackage[english]{babel}
 \usepackage{graphicx}
 \usepackage{soul}
 \usepackage{dsfont}
\usepackage{bbm}
 \usepackage{xcolor}
\usepackage[justification=centering]{caption}
 \usepackage[utf8]{inputenc}
 \usepackage[T1]{fontenc}
 \usepackage{lmodern}
 \usepackage[normalem]{ulem}
 \usepackage{verbatim}
 \usepackage{bbm}
 \usepackage{ntheorem}
 \usepackage{stmaryrd}
 \usepackage{graphics}
 \usepackage{amsmath}
 \usepackage{enumerate}
 \usepackage{amssymb}
 \usepackage[mathscr]{eucal}
 \usepackage{mathtools}
\usepackage{hyperref}
\usepackage{changepage}
\usepackage{chngcntr}

\makeatletter
\newsavebox\saved@arstrutbox
\newcommand*{\setarstrut}[1]{%
  \noalign{%
    \begingroup
      \global\setbox\saved@arstrutbox\copy\@arstrutbox
      #1%
      \global\setbox\@arstrutbox\hbox{%
        \vrule \@height\arraystretch\ht\strutbox
               \@depth\arraystretch \dp\strutbox
               \@width\z@
      }%
    \endgroup
  }%
}
\newcommand*{\restorearstrut}{%
  \noalign{%
    \global\setbox\@arstrutbox\copy\saved@arstrutbox
  }%
}
\makeatother

\makeatletter
\renewcommand\xleftrightarrow[2][]{%
  \ext@arrow 9999{\longleftrightarrowfill@}{#1}{#2}}
\newcommand\longleftrightarrowfill@{%
  \arrowfill@\leftarrow\relbar\rightarrow}
\makeatother

\theoremheaderfont{\scshape}
\theoremseparator{.}
\newtheorem{theorem}{Theorem}[section]

\newtheorem{lemma}[theorem]{Lemma}

\newtheorem{question}[theorem]{Question}

\newtheorem{coro}[theorem]{Corollary}

\newtheorem{fact}[theorem]{Fact}
\newtheorem*{nothm}[theorem]{Unproved assertion}
\newtheorem{proposition}[theorem]{Proposition}

\newenvironment{proof}[1][\proofname]{  
\noindent  {\it #1.~}%
}{
    \hfill\sqw\vspace*{.5em}
}

\newcommand{\proofname}{Proof}

\newcommand{\nlr}[1][]{\overset{#1\:}{\longleftrightarrow} \kern -17pt
  \times \kern +7 pt}

\renewcommand{\P}{\mathbb{P}}

\newcommand\blfootnote[1]{%
  \let\thefootnote\relax\footnotetext{#1}
}
\mathtoolsset{showonlyrefs}

\def\sqw{\hbox{\rlap{\leavevmode\raise.3ex\hbox{$\sqcap$}}$%
\sqcup$}}

\newcommand{\N}{\ensuremath{{{\mathbb N}}}}

\newcommand{\E}{\ensuremath{{{\mathbb E}}}}

\newcommand{\Z}{\ensuremath{\mathbb Z}}

\newcommand{\R}{\ensuremath{\mathbb R}}

\newcommand{\Red}{\ensuremath{\mathfrak R}}

\newcommand{\Per}{\ensuremath{\mathfrak P}}

\newcommand{\hei}{\ensuremath{\mathfrak h}}

\newcommand{\dev}{\ensuremath{\mathfrak d}}

\newcommand{\Tfr}{\ensuremath{\mathfrak T}}

\newcommand{\Stu}{\ensuremath{\mathfrak S}}

\newcommand{\Ndeux}{\ensuremath{\mathbb N}^2}

\newcommand{\Aux}{\ensuremath{\mathfrak A}}

\newcommand{\defini}{\textbf}

\newcommand{\act}{\textbf{act}}

\newcommand{\Var}{\textbf{Var}}

\newcommand{\h}{\textbf{h}}

\newcommand{\f}{\textbf{f}}

\newenvironment{rem}[1][Remark.]{\begin{trivlist}
\item[\hskip \labelsep {\itshape #1}]}{\end{trivlist}}

\newenvironment{definition}[1][Definition.]{\begin{trivlist}
\item[\hskip \labelsep {\textsc{#1}}]}{\end{trivlist}}

\newenvironment{definitions}[1][Definitions.]{\begin{trivlist}
\item[\hskip \labelsep {\textsc{#1}}]}{\end{trivlist}}

\newenvironment{asclust}[1][Assumption on the cluster.]{\begin{trivlist}
\item[\hskip \labelsep {\textsc{#1}}]}{\end{trivlist}}

\newenvironment{procedure}[1][Procedure.]{\begin{trivlist}
\item[\hskip \labelsep {\textsc{#1}}]}{\end{trivlist}}

\newenvironment{remark}[1][Remark.]{\begin{trivlist}
\item[\hskip \labelsep {\itshape #1}]}{\end{trivlist}}

\newenvironment{remarks}[1][Remarks.]{\begin{trivlist}
\item[\hskip \labelsep {\itshape #1}]}{\end{trivlist}}

\newenvironment{notation}[1][Notation.]{\begin{trivlist}
\item[\hskip \labelsep {\textsc{#1}}]}{\end{trivlist}}

\newenvironment{notations}[1][Notation.]{\begin{trivlist}
\item[\hskip \labelsep {\textsc{#1}}]}{\end{trivlist}}

\newenvironment{prooffr}{  
\noindent  {\it Démonstration.}%
}{
    \hfill\sqw\vspace*{.5em}
}

\newenvironment{proog}{  
   \noindent  {\textit{Proof of lemma~\ref{gron}.}}%
}{
    \hfill\sqw\vspace*{.5em}
}

\newcommand{\bde}{\begin{prooffr}\ }
\newcommand{\ede}{\end{prooffr}}

\date{\today}
\author{S\'ebastien \sc{Martineau}}
\title{Directed Diffusion-Limited Aggregation}
\begin{document}

\maketitle

\begin{abstract}
In this paper, we define a directed version of the Diffusion-Limited-Aggregation model. We present several equivalent definitions in finite volume and a definition in infinite volume. We obtain bounds on the speed of propagation of information in infinite volume and explore the geometry of the infinite cluster. We also explain how these results fit in a strategy for proving a shape theorem for this model.
\end{abstract}

\section*{Introduction}
\label{intro} 

Diffusion-Limited Aggregation (in short, DLA) is a statistical mechanics growth model that has been introduced in 1981 by Sander and Witten \cite{san}. It is defined as follows. A first particle --- a site of $\Z^2$ --- is fixed. Then, a particle is released ``at infinity'' and performs a symmetric random walk. As soon as it touches the first particle, it stops and sticks to it. Then, we release another particle, which will also stick to the cluster (the set of the particles of the aggregate), and so on\dots\ After a large number of iterations, one obtains a fractal-looking cluster.

\begin{figure}[h!]
\begin{center}
\includegraphics[width = 6.2 cm]{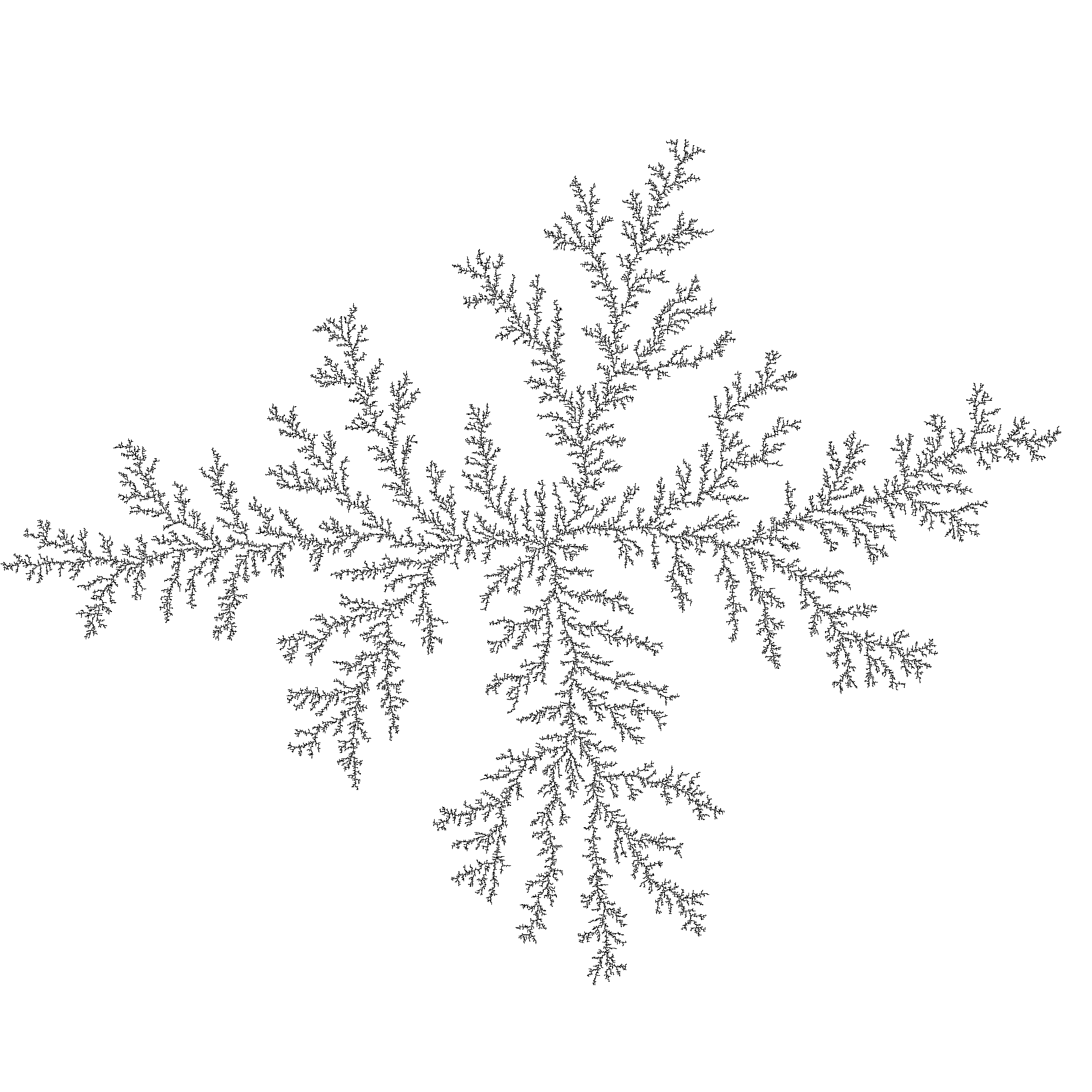}
\caption{DLA cluster obtained by Vincent {Beffara}.}
\label{dla}
\end{center}
\end{figure}

DLA does not just model sticking particles, but also Hele-Shaw flow \cite{shr}, dendritic growth \cite{vic} and dielectric breakdown \cite{bradyball}. Figure~\ref{todd} illustrates the viscous fingering phenomenon, which appears in Hele-Shaw flow. This phenomenon can be observed by injecting quickly a large quantity of oil into water.

\begin{figure}[h!]

\begin{center}
\vspace{0.5 cm}
\includegraphics[width = 8.5 cm]{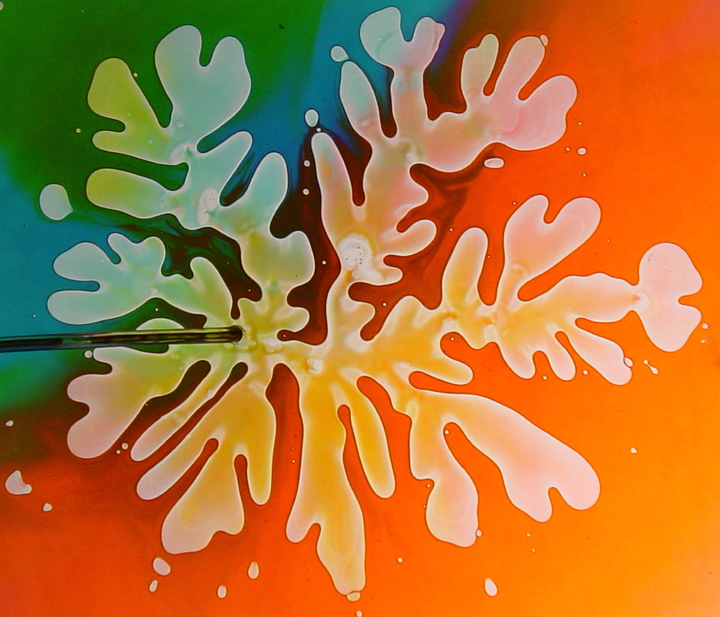}
\caption{Viscous fingering picture obtained by Jessica {Todd}.}
\label{todd}
\end{center}

\end{figure}

This model is extremely hard to study; only two non-trivial results are rigorously known about DLA: an upper bound on the speed of growth \cite{kes} and the fact that the infinite cluster has almost surely infinitely many holes, i.e.\ that the complement of the cluster has infinitely many finite components \cite{ebe}. The difficulty comes from the fact that the dynamics is neither monotone nor local, and that it roughens the cluster.

The \emph{non-locality} is quite clear: if big arms surround $P$, even if they are far from it, $P$ will never be added to the cluster.

By \emph{non-monotonicity} (which is a more serious issue), we mean that there is no coupling between a DLA starting from a cluster $C$ and another from a cluster $D\subsetneq C$ such that, at each step, the inclusion of the clusters remains valid almost surely. To understand why, throw the same particles for both dynamics, i.e.\ use the na\"ive coupling. The big cluster will catch the particles sooner than the small one: when a particle is stopped in the $C$-dynamics\footnote{and if, at the considered time, the $C$-cluster is still bigger than the $D$-one\dots}, it may go on moving for the $D$-dynamics and stick somewhere that is not in the $C$-cluster, which would break the monotonicity.
In fact, this is even a proof of the non-existence of \emph{any} monotonic coupling, under the assumption that there exists 
$(P,Q)\in D\times (C\backslash D)$ such that if $R\in\{P,Q\}$, $R$ can be connected to infinity by a $\Z^2$-path avoiding $C\backslash\{R\}$. 

Finally, the fact that the dynamics \emph{roughens} the cluster instead of smoothing it is what makes the difference between the usual (external) DLA and the internal DLA of \cite{dia}, for which a shape theorem exists \cite{law}. Even though this roughening is not mathematically established, simulations such as the one of Figure~\ref{dla} suggest it by the fractal nature of the picture they provide.

The rigorous study of DLA seeming, for the moment, out of reach, several toy models have been studied. These models are usually easier to study for one of the following reasons:
\begin{itemize}
\item either the particles are not added according to the harmonic measure of the cluster (i.e.\ launched at infinity) but ``according to some nicer measure''\footnote{See e.g.~\cite{carlesonmakarov}.};

\item or the dynamics does not occur in the plane\footnote{See e.g.~\cite{benjaminiyadin} for a study of DLA on cylinders $G\times \mathbb{N}$ or \cite{dla1d1, dla1d2, dla1d3} for results on long-range DLA on $\Z$.}.
\end{itemize}

In this paper, we prove some results on Directed Diffusion-Limited Aggregation (DDLA), which is a variant where the particles follow downward directed random walks. A large cluster is presented in Figure~\ref{ddla}. Directed versions of DLA have already been considered by physicists\footnote{See \cite{bradleystrenski84, bradleystrenski85, maj}.} but, to our knowledge, they have been rigorously studied only in the case of the binary tree (or Bethe lattice). The present model is defined in the plane.
\begin{figure}[h!]
\begin{center}
\includegraphics[angle = 135, width = 6.2 cm]{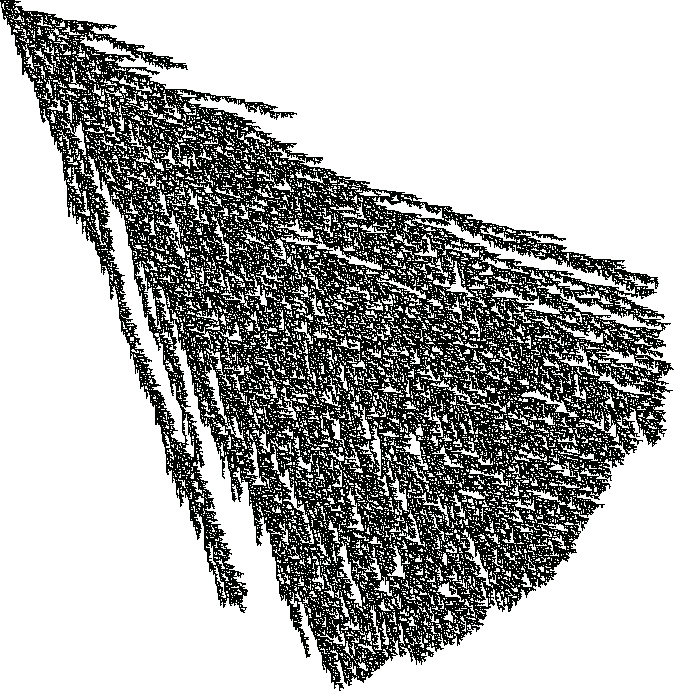}
\caption{Large DDLA cluster obtained by Vincent {Beffara}.}
\label{ddla}
\end{center}
\end{figure}
Simulations strongly suggest that the DDLA-cluster converges after suitable rescaling to some deterministic convex compact, delimited from below by two segments.

DDLA can be seen either as a ballistic deposition model where the falling particles fluctuate randomly or as a stretch version of DLA. See respectively \cite{sep} and \cite{bergerkaganprocaccia}. See also \cite{vik} for a study of the Hastings-Levitov version of DDLA; \cite{vik} and the present paper have been written independently.

\vspace{0.2 cm}

Section~\ref{defddla} is devoted to several equivalent definitions of DLA. In Section~\ref{infvol}, we define the dynamics in infinite volume. In Section~\ref{info}, we obtain a bound on the speed of propagation of the information for a DDLA starting from a (sufficiently) horizontal interface. In Section~\ref{kestin}, we adapt Kesten's argument (see \cite{kes}) to obtain bounds on the speed of horizontal growth and vertical growth. Finally, Section~\ref{sec:last} explores the geometry of the infinite cluster.

\begin{small}
\begin{notation}
We use ``a.s.e.'' as an abbreviation for ``almost surely, eventually'', which means either ``almost surely, there exists $n_0\in\N$ such that for all $n\geq n_0$'' or ``almost surely, there exists $t_0\in\R_+$ such that for all $t\geq t_0$''.
\end{notation}
\end{small}

\subsection*{Acknowledgements}

I thank Vincent {Beffara} for proposing this topic of research to me, as well as for his advice and availability.
I am indebted to Vincent Beffara and Jessica {Todd} for allowing me to use in this paper some of their pictures.

\section{Presentation of DDLA}
\label{defddla}
\subsection{Some notation}

In this paper, when dealing with DDLA, we will think of $\mathbb{Z}^2$ as rotated by an angle of $+\frac{\pi}{4}$ (so that the particles we will throw move downward).  The vertices of $\mathbb{Z}^2$ will often be referred to as \defini{sites}. Let $$\textbf{E} := \{((a,b),(c,d))\in (\mathbb{Z}^2)^2:(a=c~\&~b=d+1)\text{ or }(a=c+1~\&~b=d)\}$$ be the set of the \defini{(directed) edges}; it endows $\mathbb{Z}^2$ with a structure of directed graph. We will denote by $d$ the graph-distance on $(\Z^2,\textbf{E})$, i.e.\ the $\|\cdot\|_1$-distance.
If $e = (P,Q)$ is an edge, we call $P$ the \defini{upper vertex} of $e$ and $Q$ its \defini{lower vertex}. They are referred to as $\textbf{u}(e)$ and $\textbf{l}(e)$. 

A \defini{downward directed symmetric random walk} is a Markov chain with transition probabilities $$p(P,Q) = \textbf{1}_{(P,Q) \in \textbf{E}}/2.$$ An \defini{upward directed symmetric random walk} is obtained with transition probabilities $$p(P,Q) = \textbf{1}_{(Q,P) \in \textbf{E}}/2.$$
When the starting point of a directed random walk is not specified, it is tacitly taken to be $(0,0)$.

The \defini{height} of $P=(a,b)$, denoted by $\textbf{h}(P)$, is $a+b$. Its \defini{horizontal deviation (relative to $0$)} is $\textbf{d}(P) := b-a$. The height (resp. horizontal deviation) of $P$ \defini{relative to $Q$} is $\textbf{h}(P) - \textbf{h}(Q)$ (resp. $\textbf{d}(P) - \textbf{d}(Q)$). 
If $A \subset \Z^2$, we set $$\h(A) := \sup_{P \in A} \h(P),~~\textbf{d}(A) := \sup_{P \in A} \textbf{d}(P)\text{  and  }|\textbf{d}|(A):= \sup_{P \in A} |\textbf{d}(P)|.$$ 
The \defini{line of height $n$} is $$L_n := \{(x,y)\in \mathbb{Z}_+^2:x+y=n\}.$$We also set $$L_{\leq n} := \{(x,y)\in \mathbb{Z}_+^2:x+y\leq n\}.$$ A line $L_n$ is said to be \defini{above} a set $S$ if $S \subset L_{\leq n-1}$.
Finally, if one fixes a subset $C$ of $\Z^2$, the \defini{activity} of a site $P \in \Z^2$ relative to $C$ is
$$
\textbf{act}_C(P) := \P[\forall n \in \N, P + W_n \not\in C]\cdot|\{e\in\textbf{E}~:~ \textbf{l}(e)\in C~\&~ \textbf{u}(e)=P\}|,
$$
where $(W_n)_{n \in \Z_{+}}$ is an upward directed symmetric random walk and $|~.~|$ stands for the cardinality operator. In what follows, we will consider a growing subset of $\Z^2$, called \defini{cluster}. The \defini{current activity} (or \defini{activity}) of a site $P$ will then be relative to the cluster at the considered time. The \defini{activity of the cluster} will be the sum over all sites of their activity.

\subsection{Definition in discrete time}

At time $0$, the cluster is $C_0 := \{(0,0)\}$. Assume that the cluster has been built up to time $n$, and that $C_n \subset L_{\leq n}$. To build $C_{n+1}$, choose any line $L_k$ above $C_n$. Then, independently of all the choices made so far, choose uniformly a point in $L_k$, and send a downward symmetric random walk $(W_n)$ from this point. If the walk intersects $C_n$, then there must be a first time $\tau$ when the walker is on a point of the cluster: let
$$
C_{n+1} := C_n\cup \{W_{\tau - 1}\} \subset L_{\leq n+1}.
$$
If the random walk fails to hit the cluster, we iterate the procedure $\{$choice of a starting point $+$ launching of a random walk$\}$ independently and with the same $k$, until a random walk hits the cluster, which will happen almost surely. This is obviously the same as conditioning the procedure to succeed.

The dynamics does not depend on the choices of $k$: indeed, choosing uniformly a point in $L_{k+1}$ and taking a step downward give the same measure to all the points of $L_k$ (and if a walker goes outside $\mathbb{Z}_+^2$, it will never hit the cluster). The dynamics is thus well-defined. We call this process \defini{Directed Diffusion-Limited Aggregation} (or \defini{DDLA}).

\begin{remark}
Since the process does not depend on the choices of $k$, we can take it as large as we want so that we may (informally at least) think of the particles as falling from infinity.
\end{remark}

Here is another process, which is the same (in distribution) as DDLA. We set $C_0 := \{(0,0)\}.$ Assume that we have built $C_n$, a random set of cardinality $n+1$. We condition the following procedure to succeed:

\begin{procedure}
We choose, uniformly and independently of all the choices made so far, an edge $e$ such that $\textbf{l}(e)\in C_n$. We launch an upward directed symmetric random walk from $\textbf{u}(e)$. We say that the procedure succeeds if the random walk does not touch $C_n$.
\end{procedure}

The particle added to the cluster is the upper vertex of the edge that has been chosen. Iterating the process, we obtain a well-defined dynamics. It is the same as the first dynamics: this is easily proved by matching downward paths with the corresponding upward ones.

\subsection{Definition in continuous time}
\label{conti}

We now define \defini{DDLA in continuous time}: this is the natural continuous time version of the second definition of DDLA. Let $((N^{e}_t)_{t \geq 0})_{e \in \textbf{E}}$ be a family of independent Poisson processes of intensity 1 indexed by the set of the directed edges. 
The cluster $C(0)$ is defined as $\{(0,0)\}$ and we set $T(0):=0$.

Assume that for some (almost surely well-defined) stopping time $T(n)$, the cluster $C(T(n))$ contains exactly $n$ particles. Then, wait for an edge whose lower vertex is in $C(T(n))$ to ring (such edges will be called \defini{growth-edges}). When the clock on a growth-edge $e$ rings, send an independent upward directed random walk from its upper vertex. If it does not intersect $C(T(n))$, add a particle at $\textbf{u}(e)$ and define $T(n+1)$ to be the current time. Otherwise, wait for another growth-edge to ring, and iterate the procedure.

This dynamics is almost surely well-defined for all times\footnote{i.e.\ $\sup_n T(n)$ is almost surely infinite} because it is stochastically dominated by first-passage percolation \cite{carm}. 
Markov chain theory guarantees that
$(C_n)_{n \in \mathbb{Z}_+}$ and $(C(T(n)))_{n\in\mathbb{Z}_+}$ are identical in distribution.

\begin{remark}
This definition in continuous time consists in adding sites at a rate equal to their current activity.
\end{remark}

\subsection{Some general heuristics}
Before going any further, it may be useful to know what is the theorem we are looking for and how the results presented in this paper may play a part in its proof.
In this subsection, we present \emph{highly informal heuristics that have not been made mathematically rigorous in any way yet}. They constitute a strategy for proving a shape theorem for DDLA.

\begin{nothm}
There is some convex compact $D$ of non-empty interior such that
$\frac{C(t)}t$ converges almost surely to $D$ for the Hausdorff metric. Besides, the boundary of $D$ consists in two segments and the $(-\pi/4)$-rotated graph of a concave function.
\end{nothm}

To prove such a result, the step 0 may be to prove that the width and height of the cluster both grow linearly in time, so that we would know that we use the right scaling. This would result from a stronger form of Fact~\ref{kestt}.

Provided this, one may use compactness arguments to prove that if there exists a unique ``invariant non-empty compact set'' $D$, then we have the desired convergence (to $D$). By invariance, we informally mean the following: if $t$ is large enough and if we launch a DDLA at time $t$ from $(tD)\cap\Z^2$, then $\frac{C(t+s)}{t+s}$ ``remains close'' to $D$.

This existence and uniqueness may be proved by finding a (maybe non-explicit) ordinary differential equation satisfied by the upper interface of $D$.
To do so, we would proceed in two steps. 

\subsubsection*{Step 1}

First of all, one needs to check that the upper interface is typically ``more or less'' the $(-\pi/4)$-rotated graph of a differentiable function. To do so, one would need to control fjords. Roughly speaking, we call \defini{fjord} the area delimited by two long and close arms of the cluster. Fjords are the enemies of both regularity and ``being the graph of a function''.

Here are some heuristics about fjords:
in Figure~\ref{ddla}, we observe that there are mesoscopic fjords far from the vertical axis and no such fjord close to it. We try to account for this.

\begin{definition}
We say that a site $P$ \defini{shades} a second one if it can catch particles that would go to the second site if $P$ was vacant.
\end{definition}

Assume that we have a behaviour as suggested by Figure~\ref{ddla}. If we are close to the vertical axis, the local slope is close to $0$. We will assume that, at any time, none of the two top-points of the arms delineating the considered fjord shades the other: they will thus survive (i.e.\ keep moving), following more or less upward directed random walks. By recurrence of the 2-dimensional random walk, we obtain that the two top-points will collide at some time, closing the fjord. To avoid the shading phenomenon, one needs a still unknown proper and \emph{tractable} definition of top-point. However, it seems quite reasonable to expect this phenomenon ``not to occur'' if the slope is close to $0$ because there is no initial shading.

When the slope gets higher, the shading phenomenon appears. If the slope is not too high, the ``lower top-point'' manages to survive  but it is hard for it to catch up with the upper one: this creates a fjord\footnote{Simulations suggest that this process builds fjords forming a deterministic angle with the vertical.}. If the slope is too high, the ``lower top-point'' stops catching particles: we are in the lower interface.

\subsubsection*{Step 2}

Now, we need to find an ODE satisfied by $r$, where $\alpha \mapsto r(\alpha)$ is the angular parametrization of the upper interface of $D$ and is defined on $(-\alpha_0,\alpha_0)$. We assume that $\alpha = 0$ corresponds to what we think of as the vertical.

Assume that one can launch a DDLA from an infinite line of slope $\tan(\alpha)$ (which is made possible by Section~\ref{infvol}) and define a deterministic\footnote{by ergodicity arguments} speed of vertical growth $\textbf{v}(\alpha)$. The set $D$ being invariant, $r(\alpha)\cdot\cos(\alpha)$ must be proportional to $\textbf{v}(\theta(\alpha))$, where $\tan(\theta(\alpha))$ stands for the local slope of $D$ at the neighborhood of the point defined by $\alpha$ and $r(\alpha)$.

More exactly, we have
$$
\left\{
    \begin{array}{cl}
        r(\alpha)\cdot\cos(\alpha) = c\cdot\textbf{v}(\theta(\alpha)) \\
        \tan(\alpha-\theta(\alpha)) = \frac{r'(\alpha)}{r(\alpha)}.
    \end{array}
\right.
$$
The knowledge of $\theta(\alpha_0)$ due to the previous step allows us to find $\alpha_0$. 

\vspace{0.15 cm}

\noindent\begin{small}Simulations suggest that $\alpha_0<\pi/4$; Corollary~\ref{coro:truc} is a weak result in this direction.\end{small}

\vspace{0.35 cm}

The last point that has to be checked is that the lower interface consists of two segments.  Assume that the points of the lower interface are of bounded local slope. From this and large deviation theory, one can deduce that it costs typically exponentially much for a particle to stick to the lower interface at large distance from the upper interface.\footnote{By this, we mean that, conditionally on an initial cluster, the probability that the next particle sticks to the lower interface at distance $d$ from the upper interface is lower than $e^{-\epsilon d}$, for some constant $\epsilon$.} This might allow us to compare DDLA with ballistic deposition, for which the upper interface converges to the graph of a concave function \cite{sep} and the lower interface converges to the union of two segments (use the Kesten-Hammersley Lemma \cite{smy}).

\section{DDLA in Infinite Volume}
\label{infvol}

In this section, we define Directed Diffusion-Limited Aggregation starting from a suitable infinite set. Notice that we make the trivial adjustment that the process now lives in $\Z^2$ instead of $\Ndeux$.

Here is a very informal description of the construction. Each edge has a Poisson clock and infinitely many upward directed symmetric random walks attached to it, everything being chosen independently. When a clock rings at some edge for the $k^{th}$ time, if its upper extremity is vacant and its lower one occupied, the $k^{th}$ random walk is sent and we see if it hits the current cluster or not: we add a particle if and only if the walk does not hit the cluster.

In finite volume, this is not problematic because we can (almost surely) define  the \defini{first (or next) ringing time}: since we only need to know the state of the cluster just before we send the walk, the construction is done. In the case of an infinite initial cluster, in any non-trivial time interval, there are almost surely infinitely many ringing times to consider.\footnote{This problem is essentially the same as the one that makes impossible a discrete-time construction: we cannot throw our particles the one after the other because there is no uniform probability measure on an infinite countable set.} To define the dynamics, a solution is to show that, for all $(P_0,T_0)\in\Z^2\times\R_+^\star$, what happens at $P_0$ before time $T_0$ just depends on some random finite set of edges. Indeed, in this case, we can apply the construction in finite volume. This is the idea behind {Harris}-like constructions. See e.g.~\cite{sep} for an easy {Harris}-like construction of ballistic deposition, the local and monotonic version of DDLA.

\vspace{0.3 cm}

\label{construction}
Rigourously, the construction goes as follows. Let $((N^{e}_t)_{t \geq 0})_{e \in \textbf{E}}$ be a family of independent Poisson processes of intensity 1 indexed by the set of the directed edges. 
Let $((W^{e,k}_n)_{n \in \mathbb{N}})_{e \in \textbf{E}, k \in \mathbb{N}^\star}$ be a family of independent upward directed symmetric random walks (simply referred to as random walks in this section) indexed by $\textbf{E} \times \mathbb{N}^\star$.

\begin{notations}
Let $r_{\theta}$ be the rotation of centre $(0,0)$ and angle $\theta$. For $b\in \R^\star_+$, let $$\mathcal{C}_b := r_{-\pi/4}\left(\{(x,y)\in \R^2 : |y|\geq b|x| \}\right)$$ be the \defini{$b$-cone}\label{defcone} and let
$$
\mathcal{W}_b := r_{-\pi/4}\left(\{(x,y)\in \R^2 : |y| = (b+1)x \}\right)
$$
be the \defini{$b$-wedge}.
\begin{small}(Remember that we think of $\Z^2$ as rotated by an angle of $+\pi/4$.)\end{small} When $b$ is not specified, it is taken to be equal to the $a$ introduced in the next line.
\end{notations}

\begin{figure}[h!]
\centering
\includegraphics{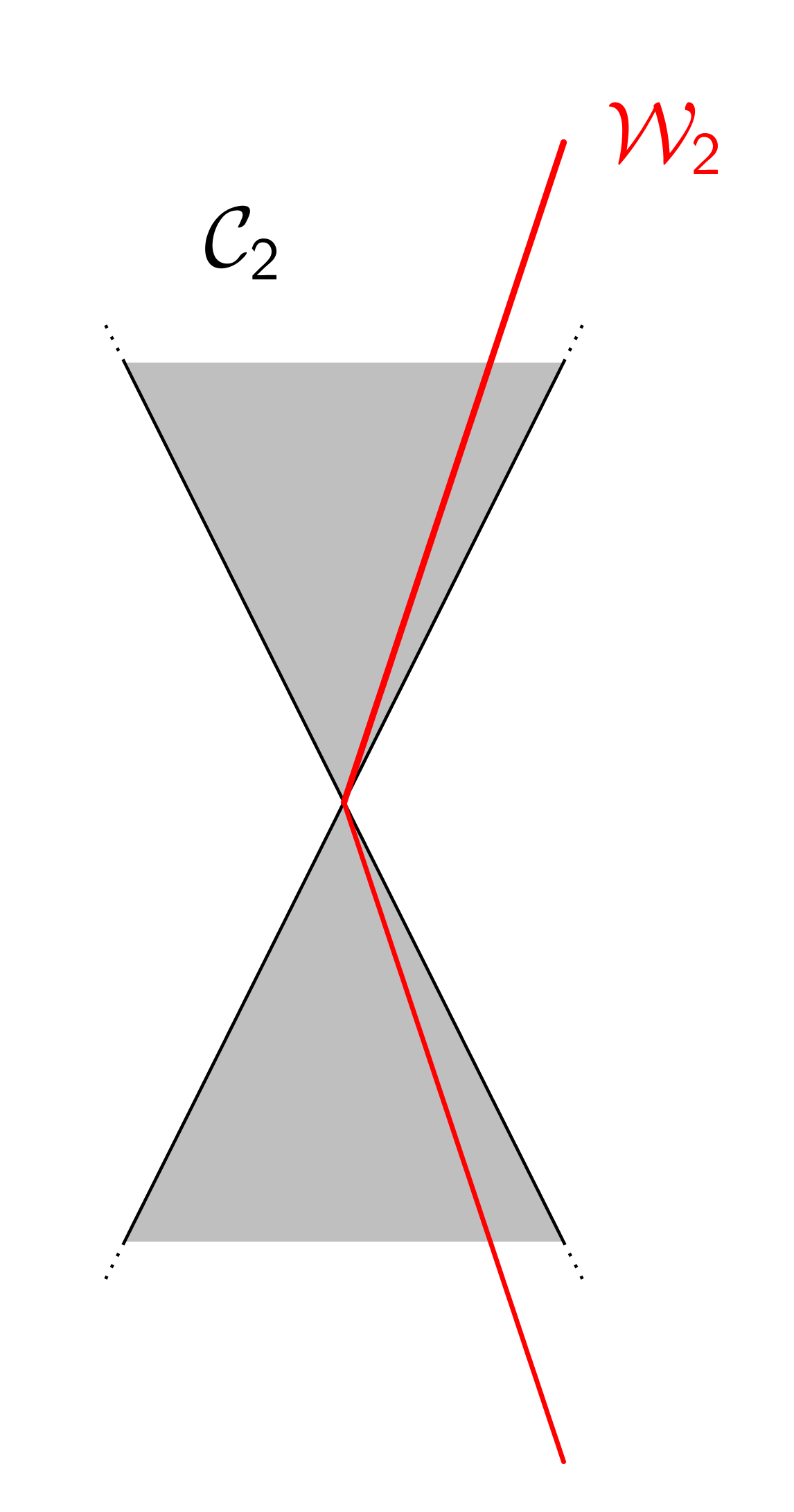}
\caption{The $b$-cone and the $b$-wedge for $b=2$.}
\end{figure}

\begin{asclust}
There is some $(a,K)\in\R_+^2$ such that for all $P\in C$,
$$
(P+(K,K)+\mathcal{C}_a)\cap C = \varnothing
~~~\text{and}~~~
(P+s((K,K)+\mathcal{C}_a))\cap C = \varnothing,
$$
where $s$ maps $Q\in\Z^2$ to $-Q$.
\end{asclust}

\begin{figure}[h!]
\centering
\includegraphics{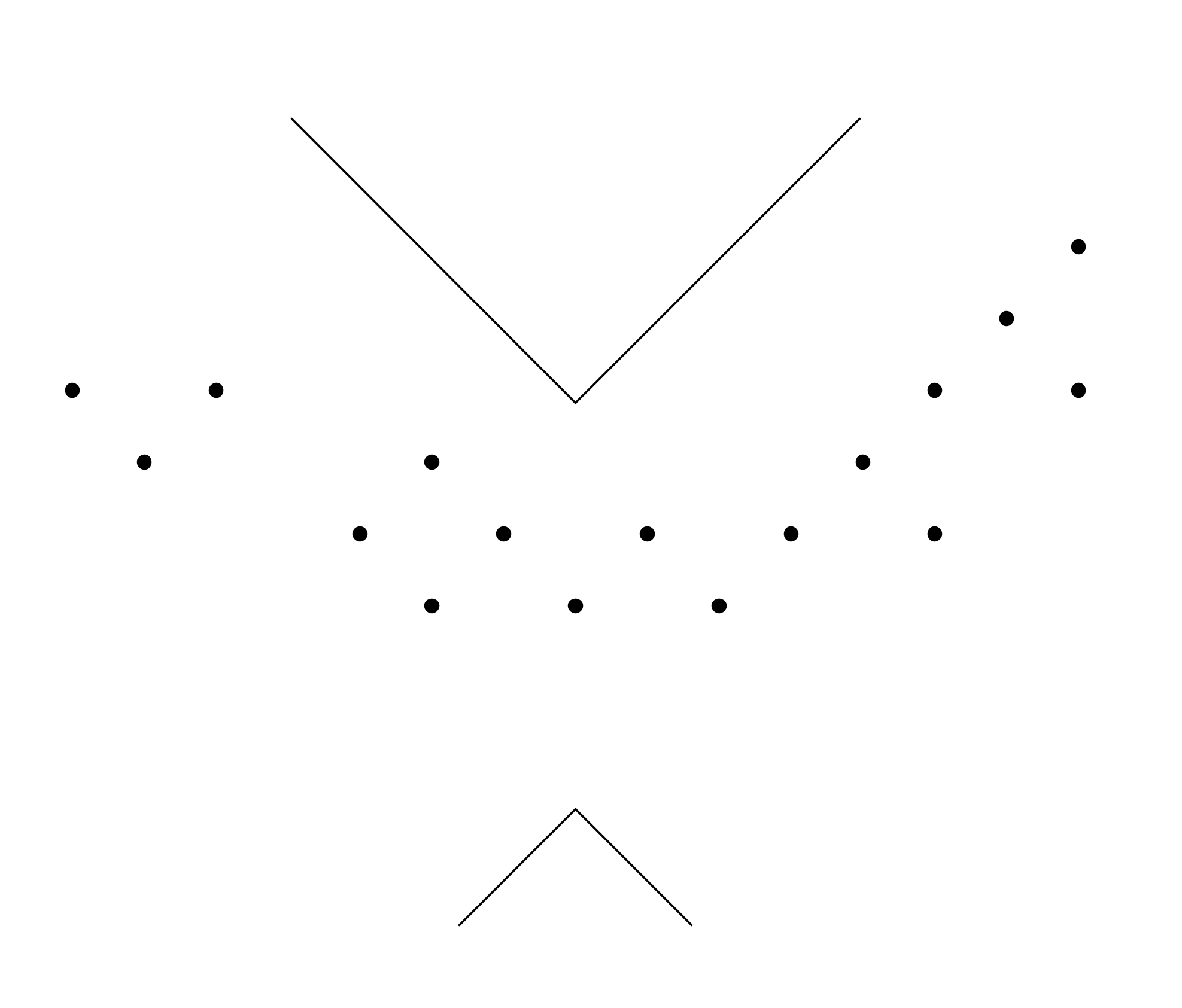}
\caption{A finite portion of a cluster satisfying the assumption for $(a,K)=(1,2)$.}
\end{figure}

Let us fix $T_0 \in\R_+^\star.$ Let us pick a site $P_0$ in $\Z^2$ and try to decide whether we add it to the cluster before time $T_0$ or not and, if so, when. If this can be done with probability 1, then the dynamics is almost surely well-defined. \begin{small}Indeed, it is enough to check every $(P_0,T_0)$ in $\Z^2\times\N^\star$.\end{small}

\begin{definitions}A site $P$ is said to be \defini{activated} if there is an upward directed path $(Q_0,\dots,Q_n)$ such that:
\begin{itemize}
\item $Q_0\in C$,
\item $Q_n=P$,
\item there is an increasing $n$-tuple $(t_1,\dots,t_n)$ such that $t_n\leq T_0$ and for every $k\in\{1,\dots,n\}$, the clock at $(Q_{k-1},Q_k)$ rings at time $t_k$.
\end{itemize}
The model consisting in adding a vertex $P$ before time $t$ if and only if the condition above is satisfied for $t$ instead of $T_0$ is called \defini{Directed First-Passage Percolation}  (or \defini{DFPP}). We also say that a directed edge $(P,Q)$ is \defini{activated} if there is an upward directed path $(Q_0,\dots,Q_n)$ such that:
\begin{itemize}
\item $Q_0\in C$,
\item $Q_{n-1}=P$
\item $Q_n=Q$,
\item there is an increasing $n$-tuple $(t_1,\dots,t_n)$ such that $t_n\leq T_0$ and for every $k\in\{1,\dots,n\}$, the clock at $(Q_{k-1},Q_k)$ rings at time $t_k$.
\end{itemize}
For any directed edge $(P,Q)$, each time the clock at $(P,Q)$ rings, if $(P,Q)$ belongs to the current DFPP cluster, then we \defini{launch} a new random walk from $Q$; the $k^\text{th}$ random walk to be launched is $Q+W^{(P,Q),k}$.
\end{definitions}

\begin{fact}
\label{fact:expo}
The probability that $P\in\Z^2$ is activated decays exponentially fast in $d(P,C)$.
\end{fact}
\begin{rem}
Fact~\ref{fact:expo} is a direct consequence of the exponential decay of subcritical percolation if $T_0 < \ln 2$.
\end{rem}
\begin{proof}
Let $B(P,n)$ denote the $\|.\|_1$-ball of centre $P$ and radius $n$. Let $k_0\in\N^\star$. If the following holds for $n = \lfloor\frac{d(P,C)}{k_0}\rfloor-1$:
$$
\begin{array}{c}
\forall~0<k \leq k_0,\\ B(P, (k_0-k+1)n)\text{ contains all the vertices that can be connected to}\\
B(P, (k_0-k)n)\text{ by edges whose clock rings between }\frac{k-1}{2}\text{ and }\frac{k}{2},
\end{array}
$$
then $P$ cannot belong to the activation cluster of $C$ for $T_0\leq k_0/2$. But, by the exponential decay of activation percolations over a time-range equal to $1/2 < \ln2$,\label{page:decroissance} the probability that this condition is not satisfied is lower than 
$$
\sum_{k=1}^{k_0} |B(P, k_0n)|ce^{-n/c},
$$
which decays exponentially fast in $n$.
\end{proof}

\begin{definitions}
Let $P\in C$. The \defini{wedge based at $P$} is defined as $P+\mathcal{W}$. It divides $\R^2$ into two connected components. A point of $\R^2$ that belongs to the same connected component as $P+(-1,1)$ is said to be to the \defini{left} of the wedge based at $P$. The set of the points of $\Z^2$ that are to the left of the wedge based at $P$ is denoted by $\mathsf{Left}(P)$. The site $P$ is said to be \defini{good} if it satisfies the following conditions:
\begin{itemize}
\item no activated directed edge $(P,Q)$ satisfies ``$P\in \mathsf{Left}(P) \Leftrightarrow Q\notin \mathsf{Left}(P)$'',
\item every random walk launched from a activated site of $\mathsf{Left}(P)$ remains in $\mathsf{Left}(P)$.
\end{itemize}
The site $P$ is said to be \defini{quasi-good} if it satisfies the following conditions:
\begin{itemize}
\item only finitely many activated edges satisfy ``exactly one extremity of the considered edge belongs to $\mathsf{Left}(P)$'',
\item only finitely many walks that are launched from an activated edge whose extremities belong to $\mathsf{Left}(P)$ do not stay in $\mathsf{Left}(P)$, and each of them takes only finitely many steps outside $\mathsf{Left}(P)$.
\end{itemize}
\end{definitions}

There is a constant $c'=c'(a,K)$ such that for every $P\in C$ and every $Q\in\Z^2$,  the following inequality holds:
$$
\max(d(Q,C), d(Q,P+\mathcal{W}))\geq c'(d(Q,P)-c').
$$
For $P\in C$, consider the following events:
\begin{itemize}
\item to an edge $(P',Q')$ such that ``$P'\in \mathsf{Left}(P) \Leftrightarrow Q'\notin \mathsf{Left}(P)$'', we associate the event ``the directed edge $(P,Q)$ is activated'',
\item to $(k,n)\in\N^\star$ and $e$ a directed edge the two extremities of which belong to $\mathsf{Left}(P)$, we associate the event ``the $k^\text{th}$ random walk at $e$ is launched and its $n^\text{th}$ step lands outside $\mathsf{Left}(P)$''.
\end{itemize}
It follows from the estimate above and large deviation theory that  the events under consideration have summable probability. The Borel-Cantelli Lemma implies that almost surely, only finitely many of these events occur: $P$ is thus almost surely quasi-good. By independence,  the site $P$ has positive probability to be good. In fact, this proof being quantitative, we know that the probability that $P$ is good can be bounded below by some positive constant $\varepsilon=\varepsilon(a,K)$.

\begin{fact}
\label{fact:goodas}
Assume that the horizontal deviation $\textbf{\emph{d}}$ is not bounded above in restriction to $C$. Then, almost surely, there is a good site $P$ such that $P_0\in\mathsf{Left}(P)$. 
\end{fact}

Taking Fact~\ref{fact:goodas} for granted, it is not hard to conclude. If \textbf{d} is bounded, then the assumption on $C$ guarantees that $C$ is finite and the process has already been defined.  We may thus assume that $C$ is infinite. If  $\textbf{d}$ is neither bounded above nor bounded below in restriction to $C$, then Fact~\ref{fact:goodas} and its symmetric version (which follows from it) imply the following: almost surely, there are a wedge to the right of $P_0$ and a (symmetric) wedge to the left of $P_0$ that are not crossed by the DFPP or any walk launched from between these wedges  before time $T_0$. Since the intersection of $C$ with the area delineated by the wedges is finite, the construction is once again reduced to finite volume. (The definition of goodness guarantees that the fate of the considered area can be defined without having to look outside it.) Finally, if $\textbf{d}$ is only bounded in one direction (say above), then one can find a site $P$ that is good and such that $P\in \mathsf{Left}(P)$: since $C\cap\mathsf{Left}(P)$ is finite, the construction in finite volume can be used.

\vspace{0.4 cm}

\begin{proof}[Proof of Fact~\ref{fact:goodas}]
Let $P$ be a point in $C$ such that $P_0\in\mathsf{Left}(P)$. (Such a point exists owing to the geometric assumption on $C$ and because $\textbf{d}$ is not bounded above in restriction to $C$.) Explore the DFPP cluster of the $1$-neighbourhood of $P+\mathcal{W}$ in reverse time: starting at time $T_0$ from $P+\mathcal{W}$, one follows the downward DFPP process associated to the same clocks. At time $t=0$, this exploration has visited a random set of sites and edges. The \defini{area explored at step 1} is this random set, together with all the vertices and edges in $\mathsf{Left}(P)$. By looking at the clocks and walks associated to this area, one can see if $P$ is good or not.
If this is the case, we stop the process. Otherwise, since we know that $P$ is \emph{quasi-}good, up to taking $P'\in C$ far enough to the right of $P$, we can assume that the information revealed so far yields no obstruction to the fact that $P'$ is good. Since we have made irrelevant all the negative information, the probability that $P'$ is good conditionally on the fact that $P$ is not good is at least the $\varepsilon$ introduced before Fact~\ref{fact:goodas}. Iterating this process, we find a good site $P$ such that $P_0\in\mathsf{Left}(P)$ in at most $k$ steps with probability at least $1-(1-\varepsilon)^k$. Thus, almost surely, such a site exists.
\end{proof}

\vspace{0.5 cm}

\begin{remarks}The dynamics is measurably defined and does not depend on the choices that are made. Besides, the $t_0$-dynamics $(t_0\in\R_+^\star)$ are coherent. More exactly, at (typical) fixed environment, if we apply the previous construction with $(P,t_0)$ and $(P,s_0)$, if the first construction says that $P$ is added at time $T<s_0$, then so does the second construction.

Also notice that this dynamics defines a simple-Markov process relative to the filtration
$$
\mathcal{F}_t := \sigma(N^e_s:s \leq t, e \in \textbf{\emph{E}}) \vee \sigma(W^{e,m}_k:e \in\textbf{\emph{E}},1\leq m \leq N^e_t,k\geq 0).
$$
\end{remarks}

\section{Transport of information}
\label{info}

In this section, we prove bounds on the speed of propagation of the information for a horizontal initial cluster. Such a control guarantees a weak (and quantitative) form of locality, which may help studying further DDLA.

\vspace{0.2 cm}
Let us consider a DDLA launched with the initial interface 
$$D := \{P \in \Z^2 ~:~ \h(P) = 0\}.$$
Before stating the proposition, we need to introduce some terminology.
Let $F \Subset \Z^2$, i.e.\ let $F$ be a non-empty finite subset of $\Z^2$. We want to define where some information about $F$ may be available. Formally, we want our area of potential influence (a random subset of $\Z^2$ depending on time) to satisfy the following property: if we use the same clocks and walks to launch a DDLA from $D$ and one from $D\Delta G$ with $G\subset F$, the clusters will be the same outside the area of potential influence at the considered time. In fact, the way this area is defined in this section, we even know that the pair {\it$($area$,$ data of the particles present in the cluster outside the area$)$} satisfies the (say weak) Markov Property.

We define this area as follows\footnote{Some looser definition may be proposed but this one is used because it is tractable.}. Instead of saying that a site of $\Z^2$ --- in the cluster or not --- belongs to the area of potential influence, we will say that it is \defini{red}, which is shorter and more visual. A non-red site belonging to the cluster will be colored in \defini{black}. Initially, 
$$
\Red_0 := F
$$
is the red area. Then, a site $P$ becomes red when one of the following events occurs:
\begin{itemize}
\item $P = \textbf{u}(e)$, the site $\textbf{l}(e)$ is red, the clock on $e$ rings and the launched random walk avoids black sites;

\item $P = \textbf{u}(e)$, the site $\textbf{l}(e)$ is black, the clock on $e$ rings and the launched random walk avoids black sites and goes through at least one red site.
\end{itemize}

It is not clear that this is well-defined, for the same reason that makes the definition in infinite volume uneasy, but we will see in the proof of Proposition~\ref{infoh} that some larger\footnote{Of course, as in all {Harris}-like constructions, this set is larger than some set that is not defined yet !} 
set is finite almost surely for all times, so that the construction boils down to finite volume, entailing proper definition of the red area.

By construction, it is clear that if it is well-defined, red is a good notion of area of potential influence.

\begin{notations}
$\Red_t$ will denote the red area at time $t$. We set $\hei_t := \h(\Red_t)$ and $\dev_t := |\textbf{d}|(\Red_t)$. This holds only for this section.
\end{notations}

\begin{proposition}
\label{infoh}
If $F \Subset \Z^2$ and if we choose $D$ as initial cluster, then $(\Red_t(F))_{t\geq 0}$ is well-defined and a.s.e.
$$
\hei_t \leq c_0\cdot t\ln t\text{  and  }\dev_t \leq c_0\cdot t^2\ln t
$$
for some deterministic constant $c_0$ independent of $F$.
\end{proposition}
\begin{proof}
Without loss of generality, we may assume that $\Red_0 = \{(0,0)\}$. \begin{small}Indeed, if one takes $F$ to be $\{(0,0)\}$, then for any finite subset $G$ of $\textbf{h}^{-1}(\N)$, the event $G\subset\Red_1$ has positive probability.\end{small}

\vspace{0.2 cm}

The rough idea of the proof is the following:
\begin{enumerate}
\item We prove that the red area cannot be extremely wide.
\item We show that if it is not very wide, it is quite small (in height).
\item We prove that if it is small, it is narrow.
\item We initialize the process with the first step and then iterate Steps 2 and Step 3, allowing us to conclude.

\end{enumerate}

\subsubsection*{Step 1: At most exponential growth}

For $n\in\N$, we set
$$
S_n:=\left\{P \in \Z^2 : \h(P) + |\textbf{d}(P)| \leq 2n\text{ and }\h(P) > 0 \right\}.
$$
We consider the following model.

\begin{figure}[h!]
\vspace{0.4 cm}
\begin{center}
\includegraphics[width = 8.2 cm]{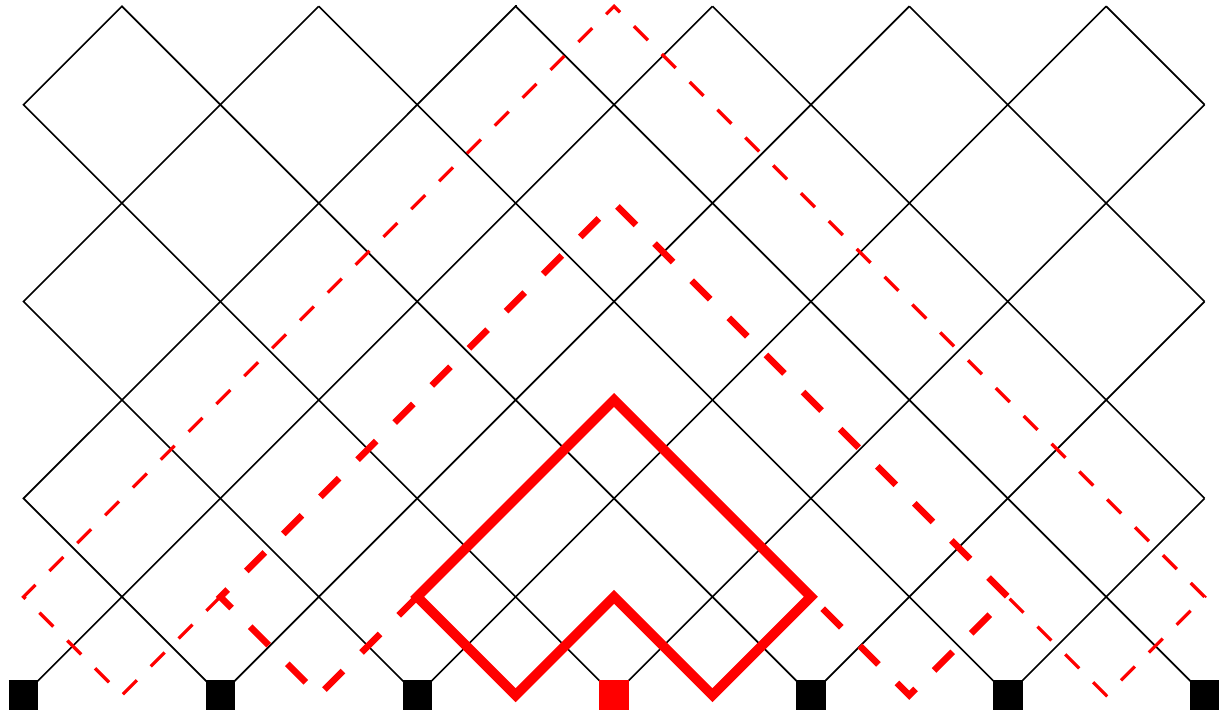}
\caption{Illustration of the model used in Step~1.}
\end{center}
\end{figure}

At time 0, the cluster is $\Stu_0 := S_0$. An edge $e$ is said to be \defini{decisive} if $\textbf{l}(e) \in \Stu_t$  and $\textbf{u}(e) \not \in \Stu_t$. The cluster does not change until a clock on a decisive edge rings. When this event occurs, $\Stu_t$, which was $S_{ n}$ for some random $n$, becomes $S_{ n+1}$. The data of $\Stu_t$ is thus just the data of this random $n(t)$.

Let $(\tau_n)_{n \geq 1}$ be a sequence of independent random variables such that $\tau_n$ follows an exponential law of parameter $2n$. Let $T_n := \sum_{k=1}^{n} \tau_k$. Then, by construction, $(T_n)_{n \in \N}$ has the same law as the sequence of the jumping times of the cluster from one state to another.

\begin{fact}
Almost surely, eventually, $$T_{\lfloor e^{n^2} \rfloor} > \frac{n^2}{8}.$$
\end{fact}
\begin{proof}
Consider $f : n \mapsto \lfloor e^{n^2}\rfloor$. By construction, one has the following estimate:
$$\E\left[\sum_{k = f(n) + 1}^{f(n+1)}~ \tau_k\right] = \sum_{k = f(n) + 1}^{f(n+1)} \frac{1}{2k} \underset{n \to \infty}{\sim} n.$$
Setting $\mathfrak{T}_n := \sum_{k = f(n) + 1}^{f(n+1)}~ \tau_k$, we have
$$
\Var[\Tfr_n] \underset{\text{indep.}}{=} \sum_{k = f(n) + 1}^{f(n+1)} \Var[\tau_k] \leq \frac{1}{4}\times\frac{\pi^2}{6}.
$$
By Chebyshev's inequality and our control on the expectation, for $n$ large enough,
$$
\P\left[\Tfr_n < \frac{n}{2}\right]  \leq \frac{\pi^2}{3n^2}.
$$
By the Borel-Cantelli Lemma, a.s.e.\  $\Tfr_n > \frac{n}{2}$. The result follows.
\end{proof}

Consequently, for some (explicit) $c \in \R_+^\star$, a.s.e.\ $\Stu_t \subset S_{\lfloor e^{ct}\rfloor}$. The area $\Red_t$ is therefore well-defined and is a.s.e.\ a subset of $S_{\lfloor e^{ct}\rfloor}$.

\subsubsection*{Step 2: Polynomial growth of the height}

\begin{lemma}\label{lem:boot}
Let $M$ be a sequence of positive real numbers such that a.s.e., $\Red_n \subset S_{ \lfloor M_n\rfloor}$. Assume that $M_n$ is eventually larger than $n$. Then for some constant $a\in\R_+^\star$, a.s.e.,
$
\hei_n \leq an\ln M_n.
$
\end{lemma}
\begin{proof}
The \defini{colored area} is the set the sites that are red or black. It is dominated by the directed first-passage percolation starting from $D$ and using the same clocks. Let $\Per_t$ be the cluster of this percolation at time $t$.
We know that, a.s.e.\ $\Red_n \subset S_{\lfloor M_n\rfloor} \cap  \Per_n =: \Aux_n^{\exp(cn)}$, where $\Aux_t^r:=S_{\lfloor r\rfloor}\cap\Per_t$.
For $n \in \N$ and $a \in \R_+^\star$,

\begin{small}
\begin{eqnarray*}
\P\left[ \h(\Aux_n^{M_n}) > an\ln M_n \right] & \leq& \P\left[\exists k \leq 2n, \h\left(\Aux_{\frac{k+1}{2}}^{M_n}\right) - \h\left(\Aux_{k/2}^{M_n}\right) > a\ln M_n/2\right]\\
 & \leq& 2n \max_{k \leq 2n} \P\left[\h\left(\Aux_{\frac{k+1}{2}}^{M_n}\right) - \h\left(\Aux_{\frac{k}{2}}^{M_n}\right) > a\ln M_n/2\right]\\
& \leq &2ne^{-\text{cst}\cdot a\ln M_n}(2M_n+1)\\
&\leq&2n (2M_n+1)^{1-\text{cst}\cdot a}.
\end{eqnarray*}
\end{small}

\begin{small}(For the last inequality, see page~\pageref{page:decroissance}.)\end{small}
Since $n=O(M_n)$, taking $a$ large enough implies that the probabilities $\P\left[ \h(\Aux_n^{M_n}) > an\ln M_n \right] $ are summable. Applying the Borel-Cantelli Lemma, we obtain that a.s.e.\ $\hei_n \leq an\ln M_n$.

\end{proof}

Applying Lemma~\ref{lem:boot} to $(e^{cn})$ and increasing slightly the value of $a$, one gets that a.s.e., $\hei(t)\leq at^2$. \begin{small}Indeed, $(n+1)^2\underset{n\to\infty}{\sim}n^2$.\end{small}

\subsubsection*{Step 3: Polynomial lateral growth }

\begin{lemma}
\label{lem:strap}
Let $M$ be a sequence of  real numbers greater than 1 such that a.s.e., $\hei_n \leq M_n$.
Then, for some constant $b\in\R_+^\star$, a.s.e.,
$
\dev_n \leq b\cdot n M_n.
$
\end{lemma}

\begin{notation}
If $k\in\N$, let $H_k := \{P \in \Z^2~:~0 \leq \h(P)\leq k \}$ be the \defini{$k$-strip}.
\end{notation}

\begin{proof}
Given a natural number $k$, we consider the dynamics defined as in Step 1, but with
$$
S_{ n}^k := S_{ n} \cap H_k
$$
instead of $S_{ n}$. We denote by $\Stu_t^k$ the corresponding cluster at time $t$. As long as $\hei_t \leq k$, we have $\Red_t \subset \Stu_t^k$.

Let $\tau_n$ be i.i.d.\ random variables following an exponential law of parameter 1 and let
$$
T_n := \sum_{i=1}^n \tau_i.
$$
The sequence of the jumping times of the $H_k$-dynamics dominates stochastically $(T_n/2k)_n$.

Large deviation theory guarantees that there is some $\text{cst}'$ such that for any $n\in\N$,
$$
\begin{array}{lll}
\P\left[ \textbf{d}(\Stu_n^{M_n}) \geq \lfloor 3nM_n\rfloor \right]
   & \leq \P\left[T_{\lfloor 3nM_n\rfloor}/(2M_n) \leq n\right]\\
{} & \leq \P\left[T_{\lfloor 3nM_n\rfloor} \leq 2nM_n\right]\\
{} & \leq e^{-\text{cst}'\times nM_n}.
\end{array}
$$
The Borel-Cantelli Lemma thus gives: a.s.e., $\textbf{d}(\Stu_n^{M_n}) < \lfloor 3nM_n\rfloor$.
\end{proof}

It results from Lemma~\ref{lem:strap} applied to the estimate of Step~2 that a.s.e., $\dev_t \leq bt^3$.

\subsubsection*{Step 4: Final bounds}

Applying Lemma~\ref{lem:boot} to the polynomial estimates we now have yields the following: a.s.e., $\hei_t \leq \text{cst}''\cdot t\ln t$. Applying Lemma~\ref{lem:strap} to this estimate gives the almost quadratic bound on the width.
\end{proof}

\begin{remark}
The same arguments can be adapted to prove Proposition~\ref{infoh} for any sufficiently horizontal initial cluster. More exactly, it is enough to assume that the initial cluster satisfies the assumption of section~\ref{infvol} for $a<1$. In this case, the constant $c_0$ depends on $a$, the quantity $\hei_t$ stands for the maximal distance from a point of $\Red_t$ to $C$ and $\dev_t$ designates the diameter of $\Red_t$.
\end{remark}

\section{Bounds on the height and width of the cluster}

\label{kestin}

Let us consider the discrete-time dynamics starting from $(0,0)$. In this section, let $\hei_n := \h(C_n)$ and $\dev_n := |\textbf{d}|(C_n)$.
Following~\cite{kes}, we obtain the following bounds:

\begin{proposition}
\label{kesprop}
For some constant $c_1$, almost surely, eventually,
$$
\sqrt{2n} \leq \hei_n \leq c_1n^{2/3}
$$
\center{and}
$$
c_1^{-1}n^{1/3} \leq \dev_n \leq c_1\sqrt{n}.
$$
\end{proposition}

\begin{remark}
For DLA, Kesten has proved that the radius of the cluster is almost surely eventually lower than $c_1n^{2/3}$.
\end{remark}
\begin{proof}
Before applying Kesten's argument, we need some lower bound on the activity of the cluster. This is natural since a high activity of the cluster guarantees, for all $P \in \Z^2$, a low probability that this site will be the next to be added to the cluster (lower than $1/\act(\text{cluster})$). This allows us to control the probability that a path of length $l$ is added between times $n_0$ and $n_1$ and thus the probability that the height (or the width) of the cluster is increased by $l$ between $n_0$ and $n_1$.

Notice that the lower bounds are consequences of the upper bounds and the fact that $C_n$ contains $n+1$ particles.

\begin{definition}
An \defini{animal} is a non-empty finite set that can be obtained by a DDLA starting from $(0,0)$.
\end{definition}

\begin{lemma}
\label{controll}
There is a constant $c$ such that
$$
\forall F \Subset \Ndeux,~ F\text{ is an animal} \implies  \textbf{\emph{act}}(F) \geq c\max(|\textbf{\emph{d}}|(F), \sqrt{\textbf{\emph{h}}(F)}).
$$
\end{lemma}
\begin{proof}
First of all, we notice that
\begin{equation}
\act(F) = \sum_{P \in L_{\h(F)}} 2\P[\exists k \in \N, P + W_k \in F]\text{,} 
\end{equation}
where $(W_k)$ is a downward directed symmetric random walk.
This is a consequence of the equivalence between the two constructions of DDLA in discrete time.

We will prove that there exists $c\in \R_+^\star$ such that for every animal $F$ and every $Q\in L_{\h(F)}$,
$$\textbf{d}(Q) \in \left[- \sqrt{\h(F)},0\right] \implies \P[\exists k \in \N, Q + W_k \in F] > c.$$ Together with the first formula of the proof, this will imply that
$$
\textbf{act}(F) \geq c\sqrt{\textbf{h}(F)}.
$$
Let $F$ be an animal and $P\in F$ be such that $\h(P) = \h(F)$. By symmetry, we can assume that $\textbf{d}(P) \geq 0$. 
Since $F$ is an animal, for all $Q\in L_{\h(F)}$ such that $\textbf{d}(Q) < \textbf{d}(P)$, we have the following inequality:
$$
\P[\exists k \in \N, Q + W_k \in F] \geq \P[\textbf{d}(Q+W_{\h(F)}) > 0].
$$
Besides, if $\textbf{d}(Q) > -\sqrt{\h(F)}$, then
$$
\P[\textbf{d}(Q+W_{\h(F)}) > 0] \geq \P\left[\tilde W_{\h(F)} >\sqrt{\h(F)}\right],
$$
where $(\tilde W_k)_k$ is the symmetric $1$-dimensional random walk. The quantity in the right-hand side of this inequality is bounded from below by the Central Limit Theorem, implying half of the desired inequality.

Let $P\in F$ be such that $|\textbf{d}(P)| = |\textbf{d}|(F)$. Since $F$ is an animal, there exists $A \subset F$ such that $A$ is an animal, $\h(P) = \h(A)$ and $P\in A$.
It follows from ($\star$) that
$$
F \subset F' \implies \act(F) \leq \act(F').
$$
Thus, we just need to prove the result for $A$.

By symmetry, we can assume that $\textbf{d}(P) > 0$.
If $Q \in L_{\h(A)}$ and $\textbf{d}(Q) \in [0,\textbf{d}(P)]$, 
$$\P[\exists k, Q+W_k \in A] \geq \P[\textbf{d}(Q + W_{\h(A)}) > 0] \geq \frac{1}2.$$
This ends the proof of the lemma.
\end{proof}

Let $(\f, \alpha)$ be $\left(n \mapsto\h(A_n), \frac{1}2\right)$ or $\left(n \mapsto|\textbf{d}|(A_n), 1\right)$. We will prove that there exists almost surely $k_0$ such that
$$
\forall k > k_0, \forall l, 2^k \leq l \leq 2^{k+1} \implies \f(2^{k+1}) - \f(l) \leq \frac{2^{k+3}}{c\f(l)^{\alpha}} + 2^{k/2}.
$$
We then conclude using the following lemma.

\begin{lemma}[discrete version of Gronwall's Lemma]
\label{gron}
Let $\alpha \in (0,1]$, $c \in \R_+^\star$ and $(a_n)_{n\in \N} \in \N^{\N}$. Assume that $\forall n, a_{n+1}-a_n \in [0,1]$ and that there exists $k_0$ such that
$$
\forall k > k_0,~ \forall l, 2^k \leq l \leq 2^{k+1} \implies \frac{a_{2^{k+1}} - a_l}{2^k} \leq \frac{8}{c\cdot a_l^{\alpha}} + 2^{-k/2}.
$$
Then, there exists some $c_1$ depending only on $(\alpha, c)$ such that, eventually, $$a_n \leq c_1 n^{1/(\alpha + 1)}.$$
\end{lemma}
Its proof is postponed to the end of the section.

Let $(k,l)$ be such that $2^k \leq l \leq 2^{k+1}$ and let us set $m := \lfloor\frac{2^{k+3}}{c\cdot \f(l)^{\alpha}} + 2^{k/2}\rfloor$.
We are looking for an upper bound on $\P[\f(2^{k+1}) - \f(l) > m]$ in order to apply the Borel-Cantelli Lemma.

\begin{definition}
The path $(P_1,\dots,P_n) \in (\Ndeux)^n$ is said to be \defini{filled in order} if
\end{definition}\begin{itemize}
\item it is an upward directed path: $\forall i,~ P_{i+1}-P_i \in \{(0,1),(1,0)\}$;

\item all the $P_i$ belong to the considered cluster;

\item if $i < j$, $P_i$ is added to the cluster before $P_j$.
\end{itemize}
Assume that $\f(2^{k+1}) - \f(l) > m$. Let $P$ be such that $$\f(P) = \max_ {Q \in C_{2^{k+1}}} \f(Q).$$ By construction of DDLA, there exists a path filled in order linking 0 to $P$. Taking its $r$ last steps for a suitable value of $r$, we obtain a path $\textbf{P} = (P_1,\dots, P_r)$ that is filled in order (relative to $C_{2^{k+1}}$) and  such that $P_r = P$ and $\f(P_r) - \f(P_1) =m$. In particular, there exists a path of length $m$ --- $(P_1,\dots, P_m)$ --- filled in order such that its sites are added to the cluster between times $l$ and $2^{k+1}$.

The number of upward directed paths of length $m$ starting in $L_{\leq l}$ is $$\frac{(l+1)(l+2)}{2}\cdot 2^m.$$ We now need to control, for such a path $\textbf{P} = (P_1, \dots, P_m)$, the probability that it is filled in order between times $l$ and $2^{k+1}$. More precisely, we extend $\textbf{P}$ to an infinite upward directed path and look for an upper bound on the probability that its first $m$ sites are successfully added between times $l$ and $2^{k+1}$. For $n \in [l+1 , 2^{k+1}]$, assume that $i = \min\{j \in \mathbb{N} : P_j\not \in C_{n-1} \}$. Let $I_n$ be the event that $P_i$ is the site added at time $n$.
The probability we want to control is lower than $\P\left[\sum_{n=l+1}^{2^{k+1}} I_n \geq m \right]$.

We know, by Lemma~\ref{controll}, that $\P[I_n|C_{n-1}] \leq \frac{1}{c\cdot \f(n-1)^{\alpha}}$. By monotonicity of $\f$, this implies that, almost surely, 
$$
\sum_{n = l+1}^{2^{k+1}} \P[I_n|C_{n-1}] \leq \frac{2^k}{c\cdot\f(l)^{\alpha}}.
$$
We now use the following exponential bound:
\begin{theorem}[Theorem 4.b in \cite{fre}]
Let $(\mathcal{F}_n)$ be a filtration. Let $\tau$ be an $(\mathcal{F}_n)$-stopping time. Let $(X_n)$ be a sequence of random variables such that
$$\text{for every }n,~ X_n \in [0,1]\text{    and    }X_n\text{ is }\mathcal{F}_n\text{-measurable}.$$ 
Let $M_n := \E[X_n|\mathcal{F}_{n-1}]$. Let $(a,b)$ be such that $0 < b\leq a$. Then,
$$
\P\left[\sum_{n=1}^{\tau} X_n \geq a\text{ and }\sum_{n=1}^{\tau}M_n \leq b\right] \leq \left(\frac{b}a\right)^a e^{a-b}.
$$
\end{theorem}
Applying this to $I_n$ with $\mathcal{F}_n := \sigma(C_0,\dots,C_n)$, $a := m$, $b:=\frac{2^k}{c\cdot\f(l)^{\alpha}}\leq \frac{m}{8}$ and a constant stopping time, we obtain
that the probability that there are at least $m$ successful fillings through \textbf{P} between times $l$ and $2^{k+1}$ 
is lower than $\left(\frac{e}{8}\right)^m$. 

Thus,
$$
\begin{array}{ll}
\P[\f(2^{k+1})- \f(l) > m]&\leq \frac{(l+1)(l+2)}{2}\cdot 2^m\cdot \left(\frac{e}{8}\right)^m\\
{}&\leq  (2^{k+1}+2)^2\cdot\left(\frac{e}{4}\right)^{2^{k/2}}.
\end{array}
$$
Since $\sum_{k\geq 1} \sum_{l=2^k+1}^{2^{k+1}} (2^{k+1}+2)^2\cdot\left(\frac{e}{4}\right)^{2^{k/2}} < \infty $, by the Borel-Cantelli Lemma and Lemma~\ref{gron}, the proposition is established.
\end{proof}
\newline
\vspace{0.2 cm}

\begin{proog}
Take $d_0$ such that $(2^{1/(\alpha +1)} - 1)\cdot d_0 > \frac{8}{cd_0^{\alpha}}+1$ and take $c_1 > 2^{1+1/(\alpha+1)}d_0$. For $k > k_0,$
$$
\begin{array}{ll}
a_{2^k} \geq d_0\cdot 2^{k/(\alpha +1)} & \implies a_{2^{k+1}} - a_{2^k} \leq \frac{8}c. \frac{2^k}{d_0^{\alpha}\cdot2^{k\alpha/(\alpha +1)}} + 2^{k/2}\\
{}&\implies a_{2^{k+1}} - a_{2^k} \leq (\frac{8}{cd_0^{\alpha}}+1)\cdot 2^{k/(\alpha+1)}\\
{}&\implies a_{2^{k+1}} - a_{2^k} \leq d_0\cdot 2^{(k+1)/(\alpha+1)} - d_0\cdot 2^{k/(\alpha + 1)},
\end{array}
$$
where the last line results from the choice of $d_0$. Thus, there exists $k_1 > k_0$ such that $a_{2^{k_1}} \leq 2d_0\cdot 2^{k_1/(\alpha+1)}$. 

If $\forall k \geq k_1,~ a_{2^k} > d_0\cdot2^{k/(\alpha+1)}$, then the implication we have just proved shows that
$$
\forall k > k_1, a_{2^k}-a_{2^{k_1}}\leq d_0\cdot2^{k/(\alpha +1)} - d_0\cdot2^{k_1/(\alpha+1)},
$$ 
which implies that $\forall k \geq k_1,~ a_{2^k} \leq 2d_0\cdot2^{k/(\alpha+1)}$. Since $(a_n)_{n\in \N}$ is a non-decreasing sequence, we obtain
$$
\forall m > 2^{k_1},~ a_m \leq 2^{1+1/(\alpha+1)}\cdot d_0\cdot m^{1/(\alpha +1)}.
$$
Thus, we can assume that $k_1$ is such that $a_{2^{k_1}} \leq d_0\cdot 2^{k_1/(\alpha+1)}$. Assume that there exists $k_2 > k_1$ such that  $a_{2^{k_2}} > d_0\cdot 2^{k_2/(\alpha+1)}$. Take a minimal such $k_2$. By minimality, there exists some minimal $l$ between $2^{k_2-1}+1$ and $2^{k_2}$ such that $a_{l-1} \leq d_0\cdot (l-1)^{1/(\alpha+1)}$ and $a_l > d_0\cdot  l^{1/(\alpha+1)}$. Thus,
$$
a_{2^{k_2}} - a_l \leq \frac{8}c. \frac{2^{k_2-1}}{d_0^{\alpha}\cdot l^{\alpha/(\alpha +1)}} + 2^{(k_2-1)/2}
$$
and, since $a_{l}\leq a_{l-1}+ 1$,
$$
\begin{array}{ll}
a_{2^{k_2}} &\leq d_0\cdot (l-1)^{1/(\alpha+1)} + 1 + \frac{8}c. \frac{2^{k_2-1}}{d_0^{\alpha}\cdot l^{\alpha/(\alpha +1)}} + 2^{(k_2-1)/2}\\
{}&\leq 2d_0\cdot 2^{k_2/(\alpha+1)}+1.
\end{array}
$$
In fact, we have proved that, for $k \geq k_1$,
$$
a_{2^k}\leq\ d_0\cdot 2^{k/(\alpha+1)}\implies a_{2^{k+1}} \leq 2d_0\cdot 2^{(k+1)/(\alpha+1)}+1
$$
\center{and}
$$
\begin{array}{ll}
a_{2^k} > d_0\cdot 2^{k/(\alpha+1)} & \implies a_{2^k} \leq 2d_0\cdot 2^{k/(\alpha+1)}+1\\
{}&\implies a_{2^{k+1}} \leq 2d_0\cdot 2^{(k+1)/(\alpha+1)}+1.
\end{array}
$$
This implies the proposition.
\end{proog}

We can deduce from this a version of Proposition~\ref{kesprop} for the continuous-time model. Of course, we set $\hei_t := \h(C_t)$ and $\dev_t := |\textbf{d}|(C_t)$.

\begin{fact}
\label{kestt}
For some constant $d_1$, almost surely, for every positive $\varepsilon$, eventually,
$$
(2-\varepsilon)t \leq \hei_t \leq d_1t
$$
\center{and}
$$
\frac{\sqrt{t}}{d_1} \leq \dev_t \leq d_1t.
$$
\end{fact}

\begin{proof}
The quantities $\hei_t$ and $\dev_t$ grow at most linearly because continuous-time DDLA is stochastically dominated by First-Passage Percolation. 

If the lower extremity of an edge is a highest point of the cluster, then the activity of this edge is 1. Consequently, if $T_k$ is the first time when the cluster is of height $k$, then $(T_{k+1}-T_k)_{k\in \N}$ is stochastically dominated by independent exponential random variables of parameter 2 (there exist at least 2 edges of lower extremity being a highest point of the cluster). This entails the at least linear growth of the height.

It results from this, the fact that discrete- and continuous-time DDLA define the same process and Proposition~\ref{kesprop} that the number $N(t)$ of particles in the cluster at time $t$ satisfies, for some deterministic constant $a$,
$$
N(t) \geq at^{3/2} ,
$$
almost surely eventually\footnote{because $N(t)$ goes to infinity when $t$ tends to infinity.}. This implies that, a.s.e.\ $\dev_t \geq \frac{N(t)^{1/3}}{c_1} \geq \frac{a^{1/3}}{c_1}\sqrt{t}$.

\end{proof}

\section{The infinite cluster}

\label{sec:last}
\begin{notation}
In this section, we set 
$$S_n := \{P \in \Z^2~:~\|P\|_1 = n\}\text{   and   }B_n := \{P \in \Z^2~:~\|P\|_1 \leq n\}.$$
We call \defini{elementary loop}
$$
L := \{P\in\Z^2~:~\|P\|_{\infty}=1\}.
$$
\end{notation}

We start this section with a formal definition of (undirected) DLA. 

\vspace{0.3cm}

Recall that if $F\Subset \Z^2$, the \defini{harmonic measure} of $F$ is the unique probability measure $\mu_F$ such that the following holds:

\vspace{0.2 cm}

\begin{center}
\begin{small}
\parbox{12cm}{Take any sequence $(\nu_n)$ of probability measures on $\Z^2$ satisfying
$$
\forall G \Subset \Z^2, \exists n_G, \forall n\geq n_G, \nu_n(G)=0.
$$
Take $W_n$ the symmetric (non-directed) nearest-neighbor random walk in $\Z^2$, starting at $0$. Choose independently a starting point $P$ according to $\nu_n$. If $G$ is a non-empty subset of $Z^2$, let
$$
\tau(G) = \min\{k~:~W_k\in G\},
$$
which is finite almost surely.
Then,
$$
\forall Q \in F, \P_n\left[P+W_{\tau(-P+F)}=Q\right] \xrightarrow[{n\to \infty}]{} \mu_F(\{Q\}).
$$}
\end{small}
\end{center}

In words, $\mu_F$ measures the probability that a site in $F$ is the first site of $F$ to be touched by a walk launched from very far. For more information on the harmonic measure, see \cite{spi}.

There are several equivalent\footnote{The equivalences between the following definition and the natural definitions you may think of boil down to the definition of harmonic measure and strong Markov Property for random walks.} definitions of DLA. The setting that will be convenient in this section is the following.
The first cluster is $C_0 := \{(0,0)\}\subset B_0$. Assume that the  first $n$ clusters have been built and are subsets of $B_n$. Independently of all the choices made so far, choose a point $P$ in $S_{n+2}$ according to $\mu_{S_{n+2}}$. Throw a symmetric random walk $(P+W_k)_{k\in\N}$ starting at $P$ and set
$$
C_{n+1} := \{P+W_{\tau(-P+C_n)-1}\}\cup C_n \subset B_{n+1}.
$$
This process is called \defini{Diffusion-Limited Aggregation}.\footnote{The process consisting in adding a site with probability proportional to its harmonic measure relative to $\{P\notin C_n~:~\exists Q\in C_n,~\|P-Q\|_1=1\}$ is very similar to this process, but not equal to it in distribution.}

\vspace{0.3 cm}

The following fact about $C_\infty:=\bigcup_n C_n$ is well-known.

\begin{fact}
\label{dens}
There is some $\varepsilon > 0$ such that
for all $P \in \Z^2$,
$
\P[P \not \in C_{\infty} ] \geq \varepsilon$.
\end{fact}
\begin{proof}
Let $P\in \Z^2$. We  consider our evolution temporally: we launch the first particle, look at it step after step until it sticks, before launching the second particle\dots A step is said to be \defini{critical} if the current particle is at distance 1 from $P$ \emph{and} $P$ is at distance 1 from the current cluster.

We wait for a critical step (we may wait forever). Conditionally on the fact that such a step exists, with probability $4^{-7}$, the particle tries --- immediately after the first critical step --- to visit all the points of $P+L$, say clockwise.\footnote{By this, we mean that the following $7$ steps that the particle would take if it was not hindered by the cluster are the ones making it visit $P+L$ clockwise.} Since the step is critical and $L$ has cardinality 8, the particle must stick to some particle of the cluster and the cardinality of $(P+L)\cap C_{\text{current time}}$ is increased by $1$. Doing so at the first 8 critical steps that occur\footnote{which means at every critical step if there are less than 8 of them\\ By ``the first 8 critical steps'', we mean the first critical step (which occurs for the $k_1^{\text{th}}$ particle), the first critical step of a particle different from the $k_1^{\text{th}}$ one, and so on up to 8.} prevents $P$ from being added to the cluster. The fact thus holds for $\varepsilon := 4^{-7\times 8}.$
\end{proof}

Such a proof cannot work for the directed version of DLA. Indeed, take a site $P$ with a neighbor belonging to the cluster. Even assuming that there are enough particles coming in the neighborhood of $P$, one cannot always surround $P$ by modifying a finite number of steps: for example, $(2,0)$ will never be added to the cluster before $(1,0)$ if one considers a DDLA launched from $(0,0)$. The screening effect of the particles above it can be very strong, but will never reduce its activity to $0$.

\vspace{0.3 cm}

However, we can prove the following proposition.

\begin{proposition}
\label{propnever}
Consider a DDLA starting from $\{(0,0)\}$.
With positive probability, the site $(1,0)$ is never added to the cluster.
\end{proposition}

\begin{proof}With positive probability, the first vertex to be added is $(0,1)$. Denote by $X_t$ the maximal first coordinate of an element of $\N\times\{1\}$ that belongs to the cluster at time $t$. At time $t$, the activity of $(1,0)$ is at most $2^{1-X_t}$ times the activity of $(X_t+1,1)$. \begin{small}(To see this inequality, map a directed random walk $W$ launched at $(1,0)$ that takes its first $X_t$ steps to the left to the random walk launched at $(X_t+1,1)$ that merges with $W$ as soon as $W$ enters $\N\times\{1\}$.)\end{small} Thus, conditionally on the fact that $(n,1)$ is added to the cluster before $(1,0)$, the probability that  $(1,0)$ is added to the cluster before $(n+1,1)$ is at most $2^{-n}$. Since $\prod_{n\geq 1}(1-2^{-n})$ is positive, Proposition~\ref{propnever} is established.
\end{proof}

\begin{coro}
\label{coro:truc}
Consider a DDLA starting from $\{(0,0)\}$.
Almost surely, for every $n\in\N$, only finitely many points of $\N\times\{n\}$ are added to the cluster.
\end{coro}

\begin{small}
\begin{notation}
Recall that for $b\in \R^\star_+$, we set $\mathcal{C}_b := r_{-\pi/4}\left(\{(x,y)\in \R^2 : |y|\geq b|x| \}\right)$.
\end{notation}
\end{small}

\begin{fact}
\label{fact:truc}
Consider a DDLA starting from $C\not=\varnothing$. Let
$$
C_{\infty} := \underset{t\geq 0}\bigcup C_t.
$$
Then, almost surely, for all $P \in \Z^2$, for all $b >0$, $
C_{\infty}\cap\left(P+\mathcal{C}_b\right)$ is infinite.
\end{fact}
\begin{proof}
Notice that it is enough to prove the fact with ``non-empty'' instead of ``infinite''.

There is an increasing path $\textbf{P}=(P_1,\dots,P_n)$ going from a point of $C$ to a point in $P+\mathcal{C}_b$.
The conic structure and the law of large numbers guarantee that the activity of $P_n$ is bounded away from $0$ (say larger than $c > 0$) as long as $P+\mathcal{C}_b = \varnothing$ (which we now assume).

Thus, if $k(t) := \max\{i~:~P_i \in C_t\}$ and if $k(t) < n$, then $P_{k(t)+1}$ will be added at rate at least $2^{n-k(t)}c > 0$. Indeed, a walk can reach $P_n$ from $P_{k+1}$ by using $\textbf{P}$; then, from $P_n$, it escapes with probability $c$. Consequently, $k(t)$ will almost surely take a finite time to increase its value, as long $k(t) < n$. Thus $k(\infty) = n$, and Fact~\ref{fact:truc} is established.
\end{proof}

\vspace{0.5 cm}

Let us conclude with a couple of questions.

\begin{question}
For which values of $b$ does it hold that the infinite DDLA cluster is almost surely a subset of $\mathcal{C}_b$ up to finitely many points?
\end{question}

\begin{question}
What is the distribution of the number of ends of the infinite DDLA cluster?
\end{question}

\vfill

Work realised at the Université de Genève and the ENS de Lyon. The author is currently a postdoc at the Weizmann Institute of Science. E-mail address: \texttt{sebastien.martineau@weizmann.ac.il}.

\begin{thebibliography}{99}
\bibitem[Freedman]{fre} D.\ \textsc{Freedman}, {\it Another note on the Borel-Cantelli lemma and the strong law, with the Poisson approximation as a by-product}, Annals of Probability, vol.\ 1 (6), p.\ 910-925, 1973.
\end{thebibliography}
\end{document}